\documentclass[a4paper,twoside,12pt]{article}

\usepackage{amssymb,amsfonts,amsmath,amsthm,latexsym}
\usepackage[pdftex,bookmarks,colorlinks=false]{hyperref}
\usepackage[hmargin=1.2in,vmargin=1.2in]{geometry}
\usepackage[auth-sc-lg,affil-sl]{authblk}
\usepackage{graphicx,caption,subcaption}
\newtheorem{thm}{Theorem}[section]
\newtheorem{cor}[thm]{Corollary}
\newtheorem{defn}[thm]{Definition}
\newtheorem{lem}[thm]{Lemma}

\newtheorem{ill}{Illustration}

\numberwithin{equation}{section}

\def\ni{\noindent}
\def\N{\mathbb{N}}

\def\H{\mathbb{H}}

\def\cS{\mathcal{S}}

\title{\textbf{\sc  On the Pythagorean Holes of Certain Graphs}}

\author{Johan Kok}
\affil{\small Tshwane Metropolitan Police Department\\ City of Tshwane, Republic of South Africa\\ E-mail: kokkiek2@tshwane.gov.za}

\author{N. K. Sudev}
\affil{\small Department of Mathematics\\ Vidya Academy of Science \& Technology \\ Thalakkottukara, Thrissur - 680501, India.\\ E-mail: sudevnk@gmail.com}

\author{K. P. Chithra}
\affil{\small Naduvath Mana, Nandikkara \\ Thrissur - 680301, India.\\ E-mail: chithrasudev@gmail.com}

\date{}

\begin{document}
\maketitle

\begin{abstract}
A \textit{primitive hole} of a graph $G$ is a cycle of length $3$ in $G$. The number of primitive holes in a given graph $G$ is called the primitive hole number of that graph $G$. The primitive degree of a vertex $v$ of a given graph $G$ is the number of primitive holes incident on the vertex $v$. In this paper, we introduce the notion of Pythagorean holes of graphs and initiate some interesting results on Pythagorean holes in general as well as results in respect of set-graphs and Jaco graphs.   
\end{abstract}

\ni \textbf{Key Words:} Set-graphs, Jaco Graphs, primitive hole, Pythagorean hole, graphical embodiment of a Pythagorean triple.

\vspace{0.2cm}

\ni \textbf{Mathematics Subject Classification:} 05C07, 05C20, 05C38.

\section{Introduction}

For general notations and concepts in graph theory, we refer to \cite{BM}, \cite{FH} and \cite{DBW}. All graphs mentioned in this paper are simple, connected undirected and finite, unless mentioned otherwise. Note that in the construction of a graph through steps and if the meaning of degree is clear at a step, the degree of a vertex $v$ will be denoted $d(v)$. If we refer to degree in the context of a graph $G$, we will denote it as $d_G(v)$.

A \textit{hole} of a simple connected graph $G$ is a chordless cycle $C_n$ , where $n \in  N$, in $G$. A \textit{primitive hole} of a graph $G$ (see \cite{KS1}) is a cycle of length $3$ in $G$. The number of primitive holes in a given graph $G$ is called the \textit{primitive hole number} of $G$ and is denoted by $h(G)$. 

The \textit{primitive degree} (see \cite{KS1}) of a vertex $v$ of a given graph $G$ is the number of primitive holes incident on the vertex $v$ and the primitive degree of the vertex $v$ in the graph $G$ is denoted by $d^p_G(v)$. 



\ni The notion of a set-graph has been introduced in \cite{KS2} as follows. 

\begin{defn}{\rm 
{\rm \cite{KS2}} Let $A^{(n)} = \{a_1, a_2, a_3, \ldots, a_n\}, n\in \N$ be a non-empty set and the $i$-th $s$-element subset of $A^{(n)}$ be denoted by $A_{s,i}^{(n)}$.  Now consider $\cS = \{A_{s,i}^{(n)}: A_{s,i}^{(n)} \subseteq A^{(n)}, A_{s,i}^{(n)} \ne \emptyset\}$. The \textit{set-graph} corresponding to set $A^{(n)}$, denoted $G_{A^{(n)}}$, is defined to be the graph with $V(G_{A^{(n)}}) = \{v_{s,i}: A_{s,i}^{(n)} \in \cS\}$ and $E(G_{A^{(n)}}) = \{v_{s,i}v_{t,j}:~ A_{s,i}^{(n)}\cap A_{t,j}^{(n)}\ne \emptyset\}$, where $s\ne t~ \text{or}~ i\ne j$. }
\end{defn}

The set $A^{(n)}\ne\emptyset$ and if $|A^{(n)}|$ is a singleton, then $G_{A^{(n)}}$ to be the trivial graph. Hence, we need to consider non-empty, non-singleton sets for the studies on set-graphs. 

It is proved in \cite{KS2} that any set-graph $G$ has odd number of vertices. An important property on set-graphs is that the the vertices of a set-graph $G$, corresponding to the sets of equal cardinality, have the same degree. Also, the vertices in a set-graph $G$, corresponding to the singleton subsets of $A^{(n)}$, are pairwise non-adjacent in $G$. 

The following is a relevant result on the minimal and maximal degree of vertices in a set-graph and their relation.

\begin{thm}
{\rm \cite{KS2}} For any vertex $v$ of a set-graph $G=G_{A^{(n)}}$, $2^{n-1}-1\le d_G(v)\le 2^n-2$. That is, $\Delta(G)=2\,\delta(G)$.
\end{thm} 

It has also been proved in \cite{KS2} that there exists a unique vertex $v$ in a set-graph $G_{A^{(n)}}$ having the highest possible degree. Another relevant result reported in the Corrigendum to \cite{KS2} is the following.

\begin{thm}
A set-graph $G_{A^{(n)}}, n \ge 2$ contains exactly $2n-2$ largest complete subgraphs (cliques) $K_{2^{n-1}}$.
\end{thm}





In this paper, we propose a new parameter called the number of Pythagorean holes of a graph. Further to some general results, we also discuss this parameter in respect of set-graphs and Jaco graphs.

\section{Pythagorean Holes of Graphs}

By a Pythagorean triple of positive integers, we mean an ordered triple $(a,b,c)$, where  $a<b<c$, such that $a^2+b^2=c^2$. Also, if $(a,b,c)$ is a Pythagorean triple of positive integers, then for any (positive) integer $k$, the triple $(ka,kb,kc)$ is also a Pythagorean triple. That is, we have $(ka)^2+(kb)^2=(kc)^2$.  

Some results in this paper are Pythagorean triple specific and can be generalised to general ordered triples of positive integers. Some results could be generalised to other ordered triples satisfying other number theoretic conditions as well but for now, our interest lies in the notion of Pythagorean holes in respect of a Pythagorean triple. Using the concepts of Pythagorean triples, we now introduce the notion of Pythagorean holes of a given graph $G$ as follows.
 
\begin{defn} {\rm 
Let $G$ be a non-empty finite graph and let the three vertices  $v_i, v_j, v_k$ in $V(G)$ induce primitive hole in $G$. This primitive hole is said to be a \textit{Pythagorean hole} if $(d_G(v_i), d_G(v_j), d_G(v_k))$ are Pythagorean triple. That is, if  $d^2_G(v_i) + d^2_G(v_j) = d^2_G(v_k)$. 
Let us denote the number of Pythagorean holes of a graph $G$ by $h^p(G)$.}
\end{defn}

We can easily construct a graph with a Pythagorean hole as follows. Let $(n_1,n_2,n_3)$ be a Pythagorean triple. Draw a triangle, say $C_3$, on the vertices $v_1, v_2, v_3$. We extend this triangle to a graph $G$ where $d_G(v_i)=n_i;~ 1\le i\le 3$ as follows. Attach $n_1-2$ pendant vertex to the vertex $v_1$, attach $n_2-2$  pendant vertices to the vertex $v_2$ and add $n_3-2$ pendant vertices to $v_3$ to obtain a new graph graph $G$. Here, $G$ is a unicyclic graph on $n_1+n_2+n_3-3$ vertices and edges each and has one primitive hole. Therefore, the triangle $v_1v_2v_3v_1$ is a Pythagorean hole in $G$. 

By a \textit{minimal graph} with respect to a given property, we mean a graph with minimum order and size satisfying that property. In view of this concept we introduce the following notion.

\begin{defn}{\rm 
A \textit{graphical embodiment} of a given Pythagorean triple is the minimal graph that consists of a Pythagorean hole with respect to that Pythagorean triple.}
\end{defn}

Clearly, the graph $G$ mentioned above is not the graphical embodiment of the Pythagorean triple $(n_1,n_2,n_3)$. Verifying the existence of a graphical embodiment to a given Pythagorean triple is an interesting question that leads to the following theorem.
 
\begin{thm}\label{T-GEPG1}
There exists a unique graphical embodiment for every Pythagorean triple of positive integers. 
\end{thm}
\begin{proof}
Let $(n_1,n_2,n_3)$ be a Pythagorean triple of positive integers such that $n_1<n_2<n_3$. First, draw a triangle on vertices $v_1, v_2, v_3$. Now, plot $n_1-2$ vertices and attach them to the vertex $v_1$ so that $d(v_1)=n_1$. Now, attach the $n_1-2$ vertices to $v_2$ and $v_3$ also. At this step, $d(v_2)=d(v_3)=n_1$. Now, the $n_2-n_1$ additional edges are required to be incident on the vertex $v_2$. Hence, plot new $n_2-n_1$ vertices and attach them to $v_2$ and $v_3$. Now, $d(v_2)=n_2$, as required. But, here $d(v_3)=n_2$ and additionally $n_3-n_2$ edges are to be incident on $v_3$. Hence, create new $n_3-n_2$ vertices attach them to $v_3$ so that $d(v_3)=n_3$. In the resultant graph $G$, $d_G(v_1)^2+d_G(v_2)^2=d_G(v_3)^2$. Hence, the triangle $v_1v_2v_3v_1$ is a Pythagorean hole in $G$. Clearly, this graph is the smallest graph with a Pythagorean hole corresponding to the given Pythagorean triple. Any graph other than $G$ will have more vertices than $G$. Hence, $G$ is a unique graphical embodiment of the given Pythagorean triple.
\end{proof}

The graph $G$ in Figure \ref{fig-2}  is an example for a graph containing a Pythagorean hole corresponding to a Pythagorean triple $(3,4,5)$. The graph $G$ has the minimum number of vertices (that is, $6$ vertices) required to contain a Pythagorean hole. 

\begin{figure}[h!]
\centering
\includegraphics[width=0.6\linewidth]{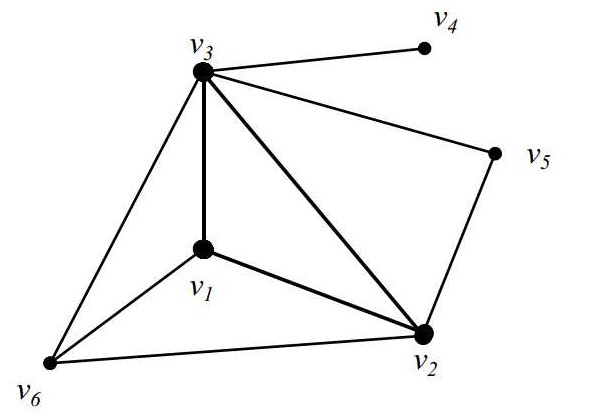}
\caption{The graphical embodiment of the Pythagorean triple $(3,4,5)$.}
\label{fig-2}
\end{figure}

From figure \ref{fig-2}, it can be noted that the graphical embodiment of the Pythagorean triple of $(3,4,5)$ has two Pythagorean holes. The uniqueness of the Pythagorean hole in a graphical embodiment of a given Pythagorean triple, other than triple $(3,4,5)$, is established in the following theorem.

\begin{thm}\label{Thm-2.4}
The graphical embodiment of any Pythagorean triple of positive integers $(n_1,n_2,n_3)\ne(3,4,5)$, consists of a unique Pythagorean hole in it.
\end{thm}
\begin{proof}
Let $(n_1,n_2,n_3)\ne (3,4,5)$ be a Pythagorean triple of positive integers and let $G$ be a graphical embodiment of this triple obtained as explained in Theorem \ref{T-GEPG1}. Let $v_1v_2v_3v_1$ be a Pythagorean hole in $G$ such that $d_G(v_1)=n_1,d_G(v_2)=n_2 ~\text{and}~ d_G(v_3)=n_3$. 

It is to be noted that all new vertices that are adjacent to $v_1$ in $G$ will be adjacent to both $v_2$ and $v_3$ also and hence are of degree $3$. Similarly, all new vertices that are adjacent to $v_2$ is adjacent to $v_3$ also. Therefore, the degree of the vertices that are adjacent to $v_2$, but not to $v_1$, is $2$ and the degree of the vertices that are adjacent only to $v_3$ is $1$. Also, no two of these new vertices are mutually adjacent. Hence, for any three vertices $v_i,v_j,v_k$, other than $v_1,v_2,v_3$, do not satisfy the condition $d_G(v_i)^2+d_G(v_j)^2=d_G(v_k)^2$. Therefore, the triangle $v_1v_2v_3$ is the unique Pythagorean hole in $G$.
\end{proof}

As stated earlier both Theorem \ref{T-GEPG1} and Theorem \ref{Thm-2.4} can be generalised to other triples with other number theoretic properties.The characteristics of the graphical embodiment of a Pythagorean triple seems to be much promising in this context. The size and order of the graphical embodiment are determined in the following result.

\begin{thm}\label{T-PGO1}
Let $G$ be the graphical embodiment of a given Pythagorean triple $(n_1,n_2,n_3)$, with usual notations. Then, 
\begin{enumerate}\itemsep0mm
\item[(i)] the order (the number of vertices) of $G$ is one greater than the highest number in the corresponding Pythagorean triple. 
\item[(ii)] the size (the number of edges) of $G$ is three less than the sum of numbers in the corresponding Pythagorean triple.  
\end{enumerate}
\end{thm}
\begin{proof}
Let $G$ be a graphical embodiment of a Pythagorean triple $(n_1,n_2,n_3)$. Let $v_1,v_2,v_3$ be the vertices with $d_G(v_1)=n_1, d_G(v_2)=n_2, ~ \text{and} ~ d_G(v_3)=n_3$. As explained in Theorem \ref{T-GEPG1}, $n_1-2$ vertices are attached to $v_1,v_2, ~ \text{and} ~  v_3$, further $n_2-n_1$ vertices are attached to $v_2  ~ \text{and} ~ v_3$ and $n_3-n_2$ pendant vertices are attached to $v+3$. Then, 
\begin{enumerate}
\item[(i)]  The number of vertices in $G$ is $|V(G)|=3+n_1-2+n_2-n_1+n_3-n_2=n_3+1$. That is, the order of the graphical embodiment is one greater than the highest number in the corresponding Pythagorean triple.

\item[(ii)] Let $V_1$ be the set of all newly introduced vertices which are adjacent to all three vertices $v_1, v_2,  ~ \text{and} ~ v_3$. Therefore, for all vertices $x$ in $V_1$, we have $d_G(x)=3$. Therefore, $\sum\limits_{x\in V_1}d_G(x)=3(n_1-2)$. Similarly, let $V_2$ be the set of new vertices which are adjacent to $v_2,  ~ \text{and} ~ v_3$. For all vertices $y$ in $V_2$, we have $d_G(y)=3$. Therefore, $\sum\limits_{y\in V_2}d_G(y)=2(n_2-n_1)$. Also, let $V_3$ be the set of new vertices that are adjacent to $v_3$ only. Here, for all vertex $z$ in $V_3$, we have $d_G(z)= 1$ and hence $\sum\limits_{z\in V_3}d_G(z)=(n_3-n_2)$. 

Therefore, $\sum\limits_{v\in V(G)}d_G(v)=n_1+n_2+n_3+3(n_1-2)+2(n_2-n_1)+(n_3-n_2) = 2(n_1+n_2+n_3-3)$. Since for any connected graph $G$, $|E(G)|=\frac{1}{2}\sum\limits_{v\in V(G)}d_G(v)$, we have $|E(G)|=n_1+n_2+n_3-3$. That is, the size of $G$ is three less than the sum of numbers in the Pythagorean triple.
\end{enumerate}
This completes the proof.
\end{proof}

\begin{thm}\label{T-PGO2}
The primitive hole number of the graphical embodiment $G$  of a Pythagorean triple is $h(G)=2n_1+n_2-5$.
\end{thm}
\begin{proof}
Let $G$ be a graphical embodiment of a Pythagorean triple $(n_1,n_2,n_3)$. Let $v_1,v_2,v_3$ be the vertices with $d_G(v_1)=n_1, d_G(v_2)=n_2, ~ \text{and} ~ d_G(v_3)=n_3$. As explained in Theorem \ref{T-GEPG1}, $n_1-2$ vertices are attached to $v_1,v_2$ and $v_3$, further $n_2-n_1$ vertices are attached to $v_2  ~ \text{and} ~ v_3$ and $n_3-n_2$ pendant vertices are attached to $v+3$. Let $V_1, V_2,V_3$ be the set of vertices as explained in the previous theorem.

Then, every vertex in $V_1$ forms a triangle with any two vertices among $v_1,v_2$ and $v_3$. Hence each vertex in $V_1$ corresponds to three triangles in $G$. Therefore, the total number of such triangles is $3(n_1-2)$. Similarly, every vertex in $V_2$, being adjacent only to $v_2$ and $v_3$, forms a triangle in $G$ and no vertex in $V_3$ is a part of a triangle in $G$, Therefore, the total number of triangles in $G$ is $h(G)=1+ 3(n_1-2)+(n_2-n_1)=2n_1+n_2-5$. This completes the proof.  
\end{proof}

The following theorem discusses certain parameters of the graphical embodiments of the given  Pythagorean triples. 

\begin{thm}\label{T-PGO3}
Let $G$ be the embodiment of a Pythagorean triple $(n_1,n_2,n_3)$. Then we have
\begin{enumerate}\itemsep0mm
\item[(i)] the chromatic number of $G$ is $\chi(G)=4$.
\item[(ii)] the independence number of $G$ is two less than the highest number in the Pythagorean triple. That is, $\alpha(G)=n_3-2$.
\item[(iii)] the covering number of $G$ is $\beta(G)=3$.
\item[(iv)] the domination number of $G$ is $\gamma(G)=1$.
\end{enumerate}
\end{thm}
\begin{proof}
Let $G$ be a graphical embodiment of a Pythagorean triple $(n_1,n_2,n_3)$ in which $v_1,v_2,v_3$ are the vertices with $d_G(v_1)=n_1, d_G(v_2)=n_2, ~ \text{and} ~ d_G(v_3)=n_3$. Let $V_1$ be the set of all new vertices that are attached to $v_1, v_2$,  and $v_3$, $V_2$ be the set of all new vertices that are attached to $v_2$,  and $v_3$ and $V_3$ be the set of new vertices that are attached only to $v_3$. Then,      

\begin{enumerate}\itemsep0mm
\item[(i)] Color $v_1$ by the color $c_1$, $v_2$ by the color $c_2$ and $v_3$ by the color $c_3$. Since all vertices in the set $V_1$ and are adjacent to the above three colors can not be used for coloring the vertices of $V_1$. Since no two vertices in $V_1$ are mutually adjacent in $G$, all vertices in $V_1$ can be colored using a fourth color $c_4$. Since all vertices in $V_2$ are mutually non-adjacent in $G$ and are not adjacent to the vertex $v_1$, we can use the color $c_1$ for coloring all the vertices in $V_2$. Since all the vertices in $V_3$ are adjacent to $v_3$ only and are pairwise non-adjacent in $G$, we can color them using either the color $c_1$ or the color $c_2$. Therefore, the minimal proper coloring of $G$ contains for colors. That is, $\chi(G)=4$.

\item[(ii)] All the newly introduced vertices in $V_1,V_2$ and $V_3$ are mutually non-adjacent and hence $V_1\cup V_2\cup V_3$ is the maximal independent set in $G$. Therefore, $\alpha(G)= n_1-2+n_2-n_1+n_3-n_2=n_3-2$.

\item[(iii)] The relation between the independence number and the covering number of a connected graph is $\alpha(G)+\beta(G)=|V(G)|$. By Theorem \ref{T-PGO1}, $|V(G)|=n_3+1$ and by previous result, $\alpha(G)=n_3-2$. Therefore, the covering number of $G$ is $n_3+1-(n_3-2)=3$.

\item[(iv)] Clearly, the vertex $v_3$ is adjacent to all other vertices in the graphical embodiment $G$, we have $\gamma(G) = 1$. 
\end{enumerate} 
This completes the proof.
\end{proof}

It is noted from Theorem \ref{T-PGO3}, the chromatic number, covering number and domination number of the graphical embodiment of any Pythagorean triple are always the same.

An immediate consequence of the general Pythagorean property is that the vertices of Pythagorean holes can be mapped onto interesting Euclidean geometric objects which we can construct along the sides of a right angled triangle or along the surfaces of the corresponding right angled prism.

\vspace{0.1cm}
\begin{ill}{\rm 
For real $x$, let $d_G(v_1) = +\sqrt{(\frac{a}{2})^2 - x^2}$,   $d_G(v_2) = +\sqrt{(\frac{b}{2})^2 - x^2}$ and  $d_G(v_3) = +\sqrt{(\frac{c}{2})^2 - x^2}$, where $(a,b,c)$ is a Pythagorean triple of positive integers. Then, 
\begin{eqnarray*}
d^2_G(v_1)  =  \int\limits_{-\frac{a}{2}}^{\frac{a}{2}}\sqrt{(\frac{a}{2})^2 - x^2}dx =  \frac{1}{2}\pi(\frac{a}{2})^2 =  \frac{1}{8}\pi a^2 \\
d^2_G(v_2)  =  \int\limits_{-\frac{b}{2}}^{\frac{b}{2}}\sqrt{(\frac{b}{2})^2 - x^2}dx =  \frac{1}{2}\pi(\frac{b}{2})^2  =  \frac{1}{8}\pi b^2 \\
d^2_G(v_3)  =  \int\limits_{-\frac{c}{2}}^{\frac{c}{2}}\sqrt{(\frac{c}{2})^2 - x^2}dx =  \frac{1}{2}\pi (\frac{c}{2})^2 =  \frac{1}{8}\pi c^2.
\end{eqnarray*}

\ni Since $\frac{1}{8}\pi$ is a constant, we have

\begin{equation*}
d^2_G(v_1)+d^2_G(v_2)= \frac{1}{8}\pi(a^2+b^2)= \frac{1}{8}\pi c^2 = d^2_G(v_3).
\end{equation*}

That is, the Pythagorean property holds for the triple $(d_G(v_1),d_G(v_2),d_G(v_3))$. Geometrically it means that the area of the semi-circle with length of hypotenuse as its diameter is equal to the sum of the respective areas of the semi-circles with lengths of the other two sides of a right angled triangle as the diameters.}
\end{ill}

\begin{ill}{\rm 
For an arbitrary real number $x$, let  $d_G(v_1) = x + bc$, $d_G(v_2) = x + ac$ and $d_G(v_3)= x + 2ab$, where $(a,b,c)$ is a Pythagorean triple of positive integers. Then, 
\begin{eqnarray*}
d^2_G(v_1) = \int_{0}^{a}(x +bc)dx = \frac{1}{2} a^2 + abc,\\
d^2_G(v_2) = \int_{0}^{b}(x +ac)dx = \frac{1}{2} b^2 + abc,\\ 
d^2_G(v_3) = \int_{0}^{c}(x + 2ab)dx = \frac{1}{2} c^2 + 2abc.
\end{eqnarray*}

Since $\frac{1}{2} a^2 + abc + \frac{1}{2} b^2 + abc = \frac{1}{2} c^2 + 2abc$, it implies geometrically that the area under the straight line $f(x) = x + 2ab, x \in \mathbb{R}$ between the limits $x=0$ and $x=c$ is equal to the sum of the respective areas under the straight lines $f(x) = x + bc, x \in \mathbb{R}$ between the limits $x=0$ and $x=a$ and $f(x) = x + ac, x \in \mathbb{R}$ between the limits $x=0$ and $x=b$ in respect of a right angled triangle.}
\end{ill}

Clearly, we can find many such mappings and it would be worthy to find the applications of these mappings.

\section{On Pythagorean Holes of Set-Graphs}

\ni The following result is on the degree sequence of set-graphs.

\begin{lem}
Consider the degree sequence $(d_1,d_2,d_3,\ldots, d_r)$ of the set-graph $G_{A^{(n)}}$. Then, for any triple $(d_i, d_j, d_k)$, where $d_i < d_j < d_k$, in that degree sequence, $d_i + d_j > d_k$.
\end{lem}
\begin{proof}
Recall the theorem that for any set-graph $G=G_{A^{(n)}}$, $2\delta(G)=\Delta(G)$. There fore, Since, $d_i+d_j>2\delta(G)=\Delta(G)\ge d_k$, the result holds for the new vertices of degree $2^n+d_i$ or $2^n-1$ as well. 
\end{proof}

\begin{thm}\label{Thm-3.2}
A set-graph has no Pythagorean holes.
\end{thm}
\begin{proof}
Note that for any Pythagorean triple $(a,b,c)$ with $a<b<c$ is either $a,b,c$ are all even, or only one of $a,b$ is even by taking modulo $4$ (see \cite{KC1}). But for any set-graph $G=G_{A^{(n)}}$, we have $\Delta(G)$ is always even and is the only vertex having even degree. Therefore, no set-graph can have a Pythagorean hole.
\end{proof}


\section{Pythagorean holes of Jaco graphs}

\begin{defn}{\rm 
The notion of the \textit{infinite Jaco graph (order 1)} was introduced in \cite{KFW} as a directed graph with $V(J_\infty(1)) = \{v_i| i \in \N\}$, $E(J_\infty(1)) \subseteq \{(v_i, v_j)| i, j \in \N, i< j\}$ and $(v_i,v_ j) \in E(J_\infty(1))$ if and only if $2i - d^-(v_i) \ge j$. We denote a finite Jaco graph by $J_n(1)$ and its underlying graph by $J^{\ast}_n(1)$. In both instances we will refer to a Jaco graph and distinguish the context by the notation $J_n(1)$ or $J^{\ast}_n(1)$.}
\end{defn}

The primitive hole number of the underlying graph of a Jaco graph has been determined in the following theorem (see \cite{KS1}).

\begin{thm}\label{T-PHNUJG}
{\rm \cite{KS1}} Let $J^{\ast}_n(1)$ be the underlying graph of a finite Jaco Graph $J_n(1)$ with Jaconian vertex $v_i$, where $n$ is a positive integer greater than or equal to $4$. Then, $h(J^{\ast}_{n+1}(1)) = h(J^{\ast}_n(1)) + \sum\limits_{j=1}^{(n-i)-1}(n-i)-j$.
\end{thm}

\ni It can be noted that the smallest Jaco graph having a Pythagorean hole is $J^\ast_8(1)$.

\begin{thm}\label{T-PHJG1}
For any primitive hole of the Jaco graph $J^{\ast}_n(1), n \in \N$ on the vertices $v_i, v_j, v_k$ with $i < j < k$ we have a primitive hole on the vertices $v_{l i}, v_{l j}, v_{l k}$ in $J^{\ast}_{n \ge lk}(1), l \in \N$ if the edge $v_{li}v_{lk}$ exists.
\end{thm}
\begin{proof}
For any primitive hole of the Jaco graph $J^{\ast}_n(1), n \in \N$ on the vertices $v_i, v_j, v_k$ with $i < j < k$ the edge $v_iv_k$ exists hence $k \le i + d^+_{J_n(1)}(v_i)$. So $l k \le l ( i + d^+_{J_n(1)}(v_i)) = l i + l  d^+_{J_n(1)}(v_i)$. Assume that the edge $v_{li}v_{lk}$ exists. Then, the subgraph induced by vertices $v_i, v_{i+1}, v_{i+2}, ... v_j, v_{j+1}, v_{j+2}, ..., v_k$ is a complete graph the same holds for the vertex $v_j$. That is, edges $v_{l i}v_{l j}$ and $v_{l j}v_{l k}$ exists. 
\end{proof}

It is known that there are $16$ primitive Pythagorean triples $(a, b, c)$ with $c \le 100$. These, in ascending order of $c$ are: $t_1 = (3,4,5)$; $(5,12,13)$; $(8,15, 17)$; $(7,24,25)$; $t_2=(20,21,29)$; $(12, 35,37)$; $(9,40,41)$; $t_3=(28, 45, 53)$; $(11,60,61)$; $(16, 63, 65)$; $(33, 56, 65)$; $t_4 = (48, 55, 73)$; $(13, 84, 85)$; $(36, 77, 85)$; $(39, 80, 89)$; $t_5 = (65, 72, 97)$.

\vspace{0.2cm}

From the definition of Jaco graphs, it easily follows that the Pythagorean triples labeled $t_i$; $1 \le i \le 5$ are applicable to the Pythagorean holes found in Jaco graphs. We shall refer to these Pythagorean triples as \textit{type i}; $i \in \N$.  With regards to Theorem \ref{T-PHJG1}, we shall refer to $lt_i = (a_il, b_il, c_il)$ as \textit{type} $i$ as well. The number of Pythagorean holes of type $i$ in a graph will be denoted $h^p_{t_i}(G)$.

Other Pythagorean triples generated by Euclid formula which are not of a specific primitive type offer some additional Pythagorean holes in Jaco graphs. Let us denote these additional types as $e_i; i \in \N$.

\ni The following results will hold for type $1$, $t_1=(3,4,5)$. 

\begin{cor}\label{C-PHJG1c}
The Jaco graph $J^{\ast}_n(1)$, where $n=5k+d^+_{J_\infty(1)}(v_{5k})$, has $h_{t_1}^p(J_n^{\ast}(1))=k$ Pythagorean holes, if $3k+d^+_{J_{\infty}(1)}(v_{3k})\ge 5k$. 
\end{cor}
\begin{proof}
The result follows from Theorem \ref{T-PHJG1}.
\end{proof}

\begin{cor}
The Jaco graphs $J^{\ast}_n(1), 8 \le n \le 15$ are the only Jaco graphs with a unique Pythagorean hole.
\end{cor}
\begin{proof}
Since $2\times 5 = 10$ and $d_{J^{\ast}_{16}(1)}(v_{10})=10$ whilst $d_{J^{\ast}_{n}(1)}(v_{10}) < 10$, for $n\le15$, the Jaco graph $J^{\ast}_{16}(1)$ has two Pythagorean holes because, $d^2_{J^{\ast}_{16}(1)}(v_3) + d^2_{J^{\ast}_{16}(1)}(v_4) = d^2_{J^{\ast}_{16}(1)}(v_5)$ and $d^2_{J^{\ast}_{16}(1)}(v_6) + d^2_{J^{\ast}_{16}(1)}(v_8) = d^2_{J^{\ast}_{16}(1)}(v_{10})$. All other Jaco graphs $J^{\ast}_{n}(1)$, for $n\geq 16$, will have at least two Pythagorean holes. Since Jaco graphs $J^{\ast}_{n, 1 \le n \le 7}(1)$ have no Pythagorasian hole, the result follows.
\end{proof}

\begin{thm}\label{T-PHJG2}
The Jaco graph $J^{\ast}_n(1), n \geq 8$ has $h^p_{t_1}(J^{\ast}_n(1)) = k$ Pythagorean holes for $5k + d^+_{J_\infty(1)}(v_{5k}) \le n < 5(k+1) + d^+_{J_\infty(1)}(v_{5(k+1)})$, alternatively $h_{t_i}^p(J^{\ast}_n(1)) = \lfloor\frac{n}{8}\rfloor$.
\end{thm}
\begin{proof}
A direct consequence of Theorem \ref{T-PHJG1} and Corollary \ref{C-PHJG1c}.
\end{proof}
The adapted Fisher Table, $J_\infty(1), 35 \geq n\in \N$ depicts the values $h(J^{\ast}_n(1))$ and $h_{t_1}^p(J^{\ast}_n(1))$.


\begin{tabular}{|c|c|c|c|c|}
\hline
$\phi (v_i)\rightarrow i\in{\Bbb{N}}$ & $d^-(v_i)=\nu(\Bbb{H}_{i-1})$ & $d^+(v_i)= i - d^-(v_i)$ & $h(J^{\ast}_i(1))$ & $h_{t_1}^p(J^{\ast}_n(1))$\\
\hline
1=$f_2$ & 0 & 1 & 0 & 0\\
\hline
2=$f_3$ & 1 & 1 & 0 & 0\\
\hline
3=$f_4$ & 1 & 2 & 0 & 0\\
\hline
4 & 1 & 3 & 0 & 0\\
\hline
5=$f_5$ & 2 & 3 & 1 & 0\\
\hline
6 & 2 & 4 & 2 & 0\\
\hline
7 & 3 & 4 & 5 & 0\\
\hline
8=$f_6$ & 3 & 5 & 8 & 1\\
\hline
9 & 3 & 6 & 11 & 1\\
\hline
10&4&6&17&1\\
\hline
11 & 4 & 7 & 23 & 1\\
\hline
12 & 4 & 8 & 29 & 1\\
\hline
13=$f_7$ & 5 & 8 & 39 & 1\\
\hline
14 & 5 & 9 & 49 & 1\\
\hline
15 & 6 & 9 & 64 & 1\\
\hline
16 & 6 & 10 & 79 & 2\\
\hline
17 & 6 & 11 & 94 & 2\\
\hline
18 & 7 & 11 & 115 & 2\\
\hline
19 & 7 & 12 & 136 & 2\\
\hline
20 & 8 & 12 & 164 & 2\\
\hline
21=$f_8$ & 8 & 13 & 192 & 2\\
\hline
22 & 8 & 14 & 220 & 2\\
\hline
23 & 9 & 14 & 256 & 2\\
\hline
24 & 9 & 15 & 292 & 3\\
\hline
25 & 9 & 16 & 328 & 3\\
\hline
26 & 10 & 16 & 373 & 3\\
\hline
27 & 10 & 17 & 418 & 3\\
\hline
28 & 11 & 17 & 473 & 3\\
\hline
29 & 11 & 18 & 528 & 3\\
\hline
30 & 11 & 19 & 583 & 3\\
\hline
31 & 12 & 19 & 649 & 3\\
\hline
32 & 12 & 20 & 715 & 4\\
\hline
33 & 12 & 21 & 781 & 4\\
\hline
34=$f_9$ & 13 & 21 & 859 & 4\\
\hline
35 & 13 & 22 & 937 & 4\\
\hline
\end{tabular}

\vspace{0.35cm}

It can be seen that Pythagorean holes are generally scares within Jaco graphs compared to the number of primitive holes. But since the graph in figure \ref{fig-2} can be extended endlessly through edge-joints to produce a graph $H$ with $h^p(H) = h(H)$, the inequality $h^p(G) \le h(G)$ holds. To construct the adapted Fisher table an improvement to Theorem \ref{T-PHNUJG} is required.

\begin{thm}
For the underlying graph $J^{\ast}_n(1)$ of a finite Jaco Graph $J_n(1), n \in \N, n \geq 4$ we have the recursion formula
$h(J^{\ast}_{n+1}(1)) = h(J^{\ast}_n(1)) + \sum\limits_{i=1}^{d^-_{J_\infty(1)}(v_{n+1})-1}i$.
\end{thm}
\begin{proof}
Consider the Jaco graph $J^{\ast}_n(1)$ with $h(J^{\ast}_n(1)) = k$. By extending from $J_n(1)$ to $J_{n+1}(1)$ a total of $d^-_{J_\infty(1)}(v_{n+1})$ arcs are added together with the vertex $v_{n+1}$. In the Jaco graph $J^{\ast}_{n+1}(1)$ vertex $v_{n+1}$ is a vertex of the new underlying Hope graph, $\H(J^{\ast}_{n+1}(1))$. Hence, with $v_{n+1}$ a common vertex to all, exactly $\sum\limits_{i=1}^{d^-_{J_\infty(1)}(v_{n+1})-1}i$, additional primitive holes are added. Hence, we have the result 
$h(J^{\ast}_{n+1}(1)) = h(J^{\ast}_n(1)) + \sum\limits_{i=1}^{d^-_{J_\infty(1)}(v_{n+1})-1}i$.
\end{proof}

\ni The following results is a direct consequence of Theorem \ref{T-PHJG1} and Theorem \ref{T-PHJG2}.

\begin{thm}\label{T-PHJG3}
Consider the distinct Pythagorean triples $t_i, t_j,\ldots, t_k, e_r, e_s, \ldots e_ x$ and no triple is a multiple of another, with the form $(a_\alpha, b_\alpha, c_\alpha)$ and $(a_\beta, b_\beta, c_\beta)$ and $c= max\{c_i, ..., c_k, c_r, ..., c_x\}$. Then, $h^p(J^{\ast}_{c+d^+(v_c)}(1)) = \sum\limits_{t_l}h^p_{t_l}(J^{\ast}_{c+d^+(v_c)}(1)) + \sum\limits_{ e_m}h^p_{e_m}(J^{\ast}_{c+d^+(v_c)}(1))$, for all $l=i,j,\ldots, k$ and $m=r,s\dots, x$.
\end{thm}

\ni The following results is a direct consequence of Theorem \ref{T-PHJG1} and Corollary \ref{C-PHJG1c}.

\begin{cor}
Let $t_i $ (or $e_m$) denotes a Pythagorean triple $(a_{\alpha},b_{\alpha},c_{\alpha})$, then $h^{p}_{t_i}(J_n^{\ast}(1))= \lfloor \frac{n}{c_{\alpha}+d^+(v_{c_{\alpha})}}\rfloor$ or $h^{p}_{e_m}(J_n^{\ast}(1))= \lfloor \frac{n}{c_{\alpha}+d^+(v_{c_{\alpha})}}\rfloor$.  
\end{cor}

\section{Conclusion and Scope for Further Studies}

We have discussed particular types of holes called Pythagorean holes of given graphs studied the existence thereof in certain graphs, particularly in set-graphs and Jaco graphs. 

As all  graphs do not contain Pythagorean holes, the study on the characteristics and structural properties of graphs containing Pythagorean holes arouses much interest. The questions regarding the number of Pythagorean holes in a given graph is a parameter that needs to be studied further. 

We proved that the graphical embodiment of all Pythagorean triples, except for $(3,4,5)$, contain exactly one Pythagorean hole. But, some graphs may contain more than one Pythagorean hole. The study on the graphs or graph classes containing more than one Pythagorean classes corresponding to one or more Pythagorean triples demands further investigation. Another interesting related area for further research is to construct the graphical embodiments of other triples and general $n$-tuples and study their characteristics.

The study seems to be promising as it can be extended to certain standard graph classes and certain graphs that are associated with the given graphs. More problems in this area are still open and hence there is a wide scope for further studies. 

It was illustrated that the vertices of a Pythagorean hole hole can be mapped onto Euclidean geometric objects. An imaginative prospect is that these theoretical applications can perhaps find real application in nano technology. It is proposed that this avenue deserves further research.



\end{document}